\documentclass[12pt]{amsart} 
\usepackage{latexsym,mathtools}    
\usepackage{mathrsfs}
\usepackage{epsfig,setspace} 
\usepackage{a4}
\usepackage{comment}
\usepackage{amssymb}
\usepackage{amsthm}
\usepackage{amsbsy} 
\usepackage{upgreek}
\usepackage{relsize}
\usepackage{tikz-cd}
\usepackage{ytableau}
\usepackage{amsfonts,amssymb,latexsym,amscd,amsmath,euscript,enumerate,verbatim,calc}
\allowdisplaybreaks%[1]
\usepackage{url}
\usepackage{hhline}
\usepackage{pifont}
\usepackage{relsize}
\usepackage[toc,page]{appendix}
\usepackage[backend=biber]{biblatex}  
\DeclareFieldFormat{doi}{\href{https://dx.doi.org/#1}{\textsc{doi}}}
\DeclareFieldFormat{url}{\href{#1}{\textsc{url}}}
\renewbibmacro*{doi+eprint+url}{
  \printfield{doi}
  \newunit\newblock
  \iftoggle{bbx:eprint}{
    \usebibmacro{eprint}
  }{}
  \newunit\newblock
  \iffieldundef{doi}{
    \usebibmacro{url+urldate}}
  {}
}
\renewbibmacro{in:}{}
\AtEveryBibitem{
  \ifentrytype{book}
  {\clearfield{pages}}
}

\usepackage[colorlinks=true,linkcolor=blue,citecolor=magenta]{hyperref}  

\theoremstyle{definition}

\newtheorem*{theorem*}{Theorem}
\newtheorem{theorem}{Theorem}[section]   
\newtheorem{lemma}[theorem]{Lemma}

\newtheorem{problem}[theorem]{Problem}
 
\newtheorem{corollary}[theorem]{Corollary} 
\newtheorem{definition}[theorem]{Definition}

\def\F{{\mathbb F}}

\def\Fq{{\mathbb F}_q}
\def\MM{{\mathrm M}}

\makeatletter
\@namedef{subjclassname@2020}{\textup{2020} Mathematics Subject Classification}
\makeatother

 \title[Simple operators and $q$-Whittaker coefficients]{Simple operators and $q$-Whittaker coefficients of power sum symmetric functions}     
 
 \author{Samrith Ram} 
\address{Indraprastha Institute of Information Technology Delhi, New Delhi, India.}
\email{samrithram@gmail.com}


\begin{document}
\begin{abstract}
We give a new proof of a theorem of Bender, Coley, Robbins and Rumsey on counting subspaces with a given profile with respect to a simple operator. Counting such subspaces is equivalent to the problem of determining the $q$-Whittaker coefficients in the expansion of the power sum symmetric function. As a consequence we obtain a result of Chen and Tseng which answers a problem of Niederreiter on splitting subspaces. 
\end{abstract}

\maketitle

% \tableofcontents
\section{Introduction}
Denote by $\Fq$ the finite field of cardinality $q$ and write $\MM_n(\Fq)$ for the algebra of $n\times n$ matrices over $\Fq$. Let ${\rm GL}_n(\Fq)$ denote the general linear group of $n\times n$ nonsingular matrices over $\Fq$. A partition of an integer $n$ is a weakly decreasing sequence $\lambda=(\lambda_1,\lambda_2,\ldots)$ of nonnegative integers with sum $n$. If $\lambda$ is a partition of $n$, we also write $\lambda\vdash n$. By convention, trailing zeroes are omitted when writing partitions. For instance, the partition $(5,2,1,1,0,\ldots)$ is considered equivalent to $(5,2,1,1)$. 
\begin{definition}
  Given a linear operator $T$ on $\Fq^n$, a subspace $W\subseteq \Fq^n$ is said to have $T$-profile $\mu=(\mu_1,\mu_2,\ldots)$ if
  \begin{align*}
    \dim(W+TW+\cdots+T^{j-1}W)=\mu_1+\mu_2+\cdots+\mu_j \mbox{ for }j\geq 1.
  \end{align*}
\end{definition}
Denote by $\sigma(\mu,T)$ the number of subspaces with $T$-profile $\mu$. Then $\sigma(\mu,T)$ depends only on the similarity class of $T$. Bender, Coley, Robbins and Rumsey~\cite{MR1141317} showed that the $T$-profile of each subspace is an integer partition and posed the following problem in 1992.
\begin{problem}\label{prob:BCRR}
  Determine $\sigma(\mu,T)$ given $\mu$ and $T$.
\end{problem}
By using Möbius inversion on the lattice of subspaces to derive equations satisfied by $\sigma(\mu,T)$, the following result was obtained in \cite{MR1141317}.
\begin{theorem}\label{thm:BCRR}
If $T$ is a simple linear operator on $\Fq^n$ (it has irreducible characteristic polynomial), then
\begin{align*}
  \sigma(\mu,T)=\frac{q^n-1}{q^{\mu_1}-1}\; q^{\sum_{j\geq 2}\mu_j^2-\mu_j}\prod_{i\geq 1}{\mu_i \brack \mu_{i+1}}_q.
\end{align*}
\end{theorem}
Here ${n\brack k}_q$ denotes a $q$-binomial coefficient, defined as the number of subspaces of $\Fq^n$ of dimension $k$. It has a well-known interpretation in geometry as the Poincaré polynomial of the Grassmannian variety ${\rm Gr}(k,n)$. In \cite{MR1141317}, the authors remark that Theorem \ref{thm:BCRR} does not appear to have a simple counting proof. The theorem above also has connections to algebraic combinatorics via $q$-Whittaker functions $W_\mu(x;q)$, class of single parameter symmetric functions which arise as specializations of Macdonald polynomials. It follows from \cite[Thm. 4.3]{ram2023subspace} that the problem of determining $\sigma(\mu,T)$ for simple $T$ is equivalent to that of finding the coefficients in the $q$-Whittaker expansion of the power sum symmetric function $p_n$. Indeed, one has \cite[Example 4.5]{ram2023subspace},
\begin{align*}
    p_n&=\sum_{\mu \vdash n}(-1)^{n-\mu_1}\frac{q^n-1}{q^{\mu_1}-1}q^{\sum_{j\geq 2}{\mu_{j} \choose 2}}\prod_{i\geq 1}{\mu_{i} \brack \mu_{i+1}}_q W_\mu(x;q).
\end{align*}
Subsequently, unaware of Theorem \ref{thm:BCRR}, Niederreiter \cite[p. 11]{MR1334623} stated the following problem in the context of his multiple recursive matrix method for pseudorandom number generation.

\begin{problem}\label{prob:nied}
  Let $m,d$ be positive integers. Given an element $\alpha\in \F_{q^{md}}$, such that $F_{q^{md}}=\Fq(\alpha)$, find the number of $m$-dimensional subspaces $W\subseteq \F_{q^{md}}$ such that
  \begin{equation*}\label{eq:split}
    \F_{q^{md}}=W\oplus \alpha W\oplus \alpha^2 W\oplus \cdots \oplus \alpha^{d-1}W.
  \end{equation*}
\end{problem}
A subspace $W$ satisfying the above equation is called an $\alpha$-splitting subspace. Niederreiter was actually interested in the primitive case, namely when $\alpha$ generates the multiplicative group of nonzero elements in $\F_{q^{md}}$ but it is quite natural to study the slightly more general problem above. The primitive case of Problem \ref{prob:nied} has connections with finite projective geometry and group theory via Singer cycles, elements of maximum possible order in ${\rm GL}_n(\Fq)$. The special case of counting $2$-dimensional splitting subspaces is closely linked to the problem of counting pairs of coprime polynomials over a finite field \cite{split}. This problem, and a generalization to the setting of combinatorial prefabs, has been studied by Corteel, Savage, Wilf and Zeilberger \cite{MR1620873} who prove a general pentagonal sieve theorem in the spirit of Euler. Chen and Tseng \cite{MR3093853} proved a conjecture \cite[Conj.~5.5]{MR2831705} that the number of $m$-dimensional $\alpha$-splitting subspaces in Problem~\ref{prob:nied} is given by 
\begin{align*}
\frac{q^{md}-1}{q^m-1}q^{m(m-1)(d-1)}, 
\end{align*}
thereby solving the problem of Niederreiter. In fact, another proof of Theorem~\ref{thm:BCRR} can be obtained by making a suitable choice of parameters in the general $q$-identity of \cite[Thm. 3.3]{MR3093853}. An alternate solution to Niederreiter's problem appears in \cite{MR4263652}. After these developments, it came to light \cite[p. 2]{ram2023diagonal} that Theorem~\ref{thm:BCRR} can be used to answer the question of Niederreiter. 

The main aim of this paper is to give another proof of Theorem \ref{thm:BCRR}. As noted above, this essentially gives another derivation of the $q$-Whittaker coefficients in the power sum symmetric function. The proof given here does not require the theory of symmetric functions; instead it relies on enumeration results involving partial transformations over finite fields and some ideas from \cite{MR4263652}. The answer to Niederreiter's problem is then deduced as a consequence in Corollary~\ref{cor:split}. The approach taken is very different from that in Bender, Coley, Robbins and Rumsey \cite{MR1141317} and Chen and Tseng \cite{MR3093853} where the counting problems are reduced to proving certain $q$-binomial identities. Finally, we remark that specific instances of Problem~\ref{prob:BCRR} have been discussed in many papers \cite{split,MR4349887, MR4555237,MR4682040,MR4797454, ram2023diagonal} with the general solution appearing in \cite{ram2023subspace}.
\section{Subspace profiles and partial maps}
We require some results on partial linear maps which we will relate to subspace profiles later on. Given a subspace $W$ of a finite-dimensional vector space $V$ and a linear map $T:W\to V$, define a sequence of subspaces (see Gohberg, Kaashoek and van Schagen \cite[Sec. III.1]{Gohbergetal1995}) $W_i(i\geq 0)$ by $W_0=V,W_1=W$ and for $i\geq 1$,
\begin{align*}
  W_{i+1}:=W_i\cap T^{-1}W_i=W\cap T^{-1}W\cap \cdots \cap T^{-i}W.
\end{align*}
Here $T^{-1}$ denotes the inverse image under $T$. The descending sequence $W_0\supseteq W_1\supseteq \cdots$ necessarily stabilizes since the dimensions of the subspaces $W_i$ are nonnegative integers. Write $d_i=\dim W_i$ for each $i\geq 0$ and define
\begin{align*}
  \ell=\ell(T):=\min \{i:W_i=W_{i+1}\}.
\end{align*}
Thus $W_\ell$ is the maximal $T$-invariant subspace. The integers $\lambda_i=\lambda_i(T):=d_{i-1}-d_i$ for $1\leq i\leq \ell$ are called the defect dimensions of $T$ \cite[p. 52]{Gohbergetal1995}. It is not difficult to see that $\lambda(T)$ is an integer partition of $n-d_\ell$ with first part $\dim V-\dim W$ \cite[Lem. 2.2]{MR4264825}.
\begin{definition}
Given a subspace $W$ of $V$, a linear map $T:W\to V$ is said to be \emph{simple} if for each $T$-invariant subspace $U$ we have $U=\{0\}$ or $U=V$. 
\end{definition}
Note that a linear operator $T$ on a finite vector space $V$ is simple precisely when $T$ has an irreducible characteristic polynomial. On the other hand, for a simple map $T$ defined on a proper subspace $W\subset V$, we have $\lambda(T)$ is a partition of $n$ since the only $T$-invariant subspace is the zero subspace.
\begin{theorem}\label{thm:simplewithdefect}\cite[Cor. 3.6]{MR4264825}
  If $W\subseteq \Fq^n$ is a proper subspace of dimension $k$ and $\mu$ is a partition of $n$ with $\mu_1=n-k$, then the number of simple maps defined on $W$ with defect dimensions $\mu$ equals
  \begin{align*}
        q^{\sum_{j\geq 2}\mu_j^2}\;\gamma_q(k)\;\prod_{i\geq 1}{\mu_i\brack \mu_{i+1}}_q,
  \end{align*}
 where $\gamma_q(k):=|{\rm GL}_k(\Fq)|=\prod_{i=0}^{k-1}(q^k-q^i)$.
\end{theorem}

Given a linear operator $T$ on $V$, let $T^*$ denote the transpose of $T$ defined on the linear dual $V^*$. For each subspace $W\subseteq V$, let $W^0$ denote its annihilator (the space of all linear functionals which vanish on $W$) in $V^*.$ The following lemma demonstrates the duality between profiles and defect dimensions. 
\begin{lemma}\label{lem:duality}
If $T$ is a linear operator on $V$, then a subspace $W$ of $V$ has $T$-profile $\mu$ if and only if the restriction of $T^*$ to $W^0$ has defect dimensions $\mu$.
\end{lemma}
\begin{proof}
Suppose $W$ has $T$-profile $\mu$ and let $\tilde{\mu}$ denote the defect dimensions of the restriction of $T^*$ to $W^0$. By definition, for each $j\geq 1$,
  \begin{align*}
    \mu_j=\dim(W+TW+\cdots+T^{j-1}W)-\dim(W+TW+\cdots+T^{j-2}W).
  \end{align*}
  On the other hand,
  \begin{align*}
    \tilde{\mu}_j&=\dim(W^0\cap T^{*-1}W^0\cap\cdots\cap T^{*-(j-2)}W^0)\\
    &\qquad\qquad-\dim(W^0\cap T^{*-1}W^0\cap\cdots\cap T^{*-(j-1)}W^0),
  \end{align*}
   where $T^{*-1}$ denotes the inverse image under $T^*$. Here an empty sum is interpreted as the zero subspace while an empty intersection is all of $V$. The annihilator of the subspace $TU$ is given by $T^{*-1}U^0$. Therefore, the annihilator of the subspace $U+TU+\cdots+T^jU$ is the subspace $U^0\cap T^{*-1}U^0\cap \cdots \cap T^{*-j}U^0$. The identity $\dim U_1-\dim U_2=\dim U_2^0-\dim U_1^0$ implies that $\mu_j=\tilde{\mu}_j$ for each $j\geq 1$.
 \end{proof}
As every operator is similar to its transpose, we obtain the following corollary.
\begin{corollary}\label{cor:duality}
  For each operator $T$ on $\Fq^n$, the number of subspaces $W\subseteq \Fq^n$ such that the restriction of $T$ to $W$ has defect dimensions $\mu$ is given by $\sigma(\mu,T)$. 
\end{corollary}

\begin{lemma}\label{lem:extension}\cite[Lem. 2.10]{MR4263652}
  Let $W\subset W'$ be subspaces of $\Fq^n$ with dimensions $k,k+1$ respectively where $0\leq k<n-1$. Given a simple map $T:W\to \Fq^n$, the number of extensions of $T$ to a simple map defined on $W'$ is given by $q^n-q^{k+1}$.
\end{lemma}

One can characterize simplicity via matrices. Given a $k$-dimensional subspace $W\subseteq \Fq^n$ let $\mathcal{B}$ be a basis for $W$ and suppose $\mathcal{B}'\supseteq \mathcal{B}$ is a basis for $\Fq^n$. Given a linear map $T:W\to \Fq^n$ let $A$ denote the $n\times k$ matrix of $T$ with respect to the bases $\mathcal{B}$ and $\mathcal{B}'$. The \emph{invariant factors} of $T$ are the diagonal entries in the Smith normal form \cite[p. 257]{MR0276251} of $xI-A$, where $I$ denotes the $n\times k$ matrix whose $(i,j)$th entry is 1 if $i=j$ and 0 otherwise. In fact, $T$ is simple precisely when the matrix polynomial $xI-A$ is unimodular (its maximal minors are coprime, equivalently its invariant factors are all equal to 1) \cite[Prop. 2.3]{MR4263652}. We require the following result of Wimmer \cite{Wimmer1974} (also see Cravo \cite[Thm. 15]{Cravo2009}). 
\begin{theorem}\label{thm:wimmer}
  Let $F$ be a field and suppose $W\subset F^n$ is a $k$-dimensional subspace. Let $T:W\to F^n$ be a linear map and suppose $f_1\mid f_2\mid\cdots\mid f_k$ are the invariant factors of $T$. Then $T$ can be extended to a linear operator on $F^n$ with characteristic polynomial $f$ if and only if $\prod_{i=1}^k f_i$ divides $f$.
\end{theorem}
\begin{lemma}\label{lem:numextensions} 
Let $f\in \Fq[x]$ be a monic irreducible polynomial of degree $n$ and suppose $W\subseteq \Fq^n$ is a $k$-dimensional subspace. If $T:W\to \Fq^n$ is simple, then the number of extensions of $T$ to all of $\Fq^n$ with characteristic polynomial $f$ is given by $\prod_{j=k+1}^{n-1}(q^n-q^j)$.
\end{lemma}

\begin{proof}
  Consider a fixed partial flag $W=W_k\subset W_{k+1}\subset \cdots \subset W_{n-1}\subset W_n=\Fq^n$ of subspaces with $\dim W_i=i$.  The number of extensions of $T$ to a simple map on $W_{n-1}$ equals $\prod_{j=k+1}^{n-1}(q^n-q^j)$ by repeated applications of Lemma \ref{lem:extension}. Given a simple map $T'$ defined on $W_{n-1}$ it can be extended in $q^n$ ways to all of $\Fq^n$. By Theorem \ref{thm:wimmer}, $T'$ can be extended to obtain any characteristic polynomial of degree $n$ (since $T'$ has all invariant factors equal to 1). Since there are $q^n$ monic polynomials of degree $n$, it follows that each such monic polynomial is the characteristic polynomial for a unique extension of $T'$. Therefore, the number of extensions of $T$ to all of $\Fq^n$ with characteristic polynomial $f$ equals $\prod_{j=k+1}^{n-1}(q^n-q^j)$.
\end{proof}
We are now ready to prove the main theorem.
\begin{theorem}\label{thm:main}
  If $T$ is a simple operator on $\Fq^n$ and $\mu$ is a partition of $n$, then
  \begin{align*}
    \sigma(\mu,T)=\frac{q^n-1}{q^{\mu_1}-1}q^{\sum_{j\geq 2}\mu_j^2-\mu_j}\prod_{i\geq 1}{\mu_i \brack \mu_{i+1}}_q.
  \end{align*}
\end{theorem}
\begin{proof}
  Let $k=n-\mu_1$ and let $f$ denote the characteristic polynomial of $T$. Double count pairs $(W,\tilde{T})$ such that $W$ is a $k$-dimensional subspace of $\Fq^n$ and $\tilde{T}$ is an operator on $\Fq^n$ with characteristic polynomial $f$ such that the restriction $\tilde{T}_W$ of $\tilde{T}$ to $W$ is simple with defect dimensions $\mu$. Since any two operators with characteristic polynomial $f$ are similar we have $\sigma(\mu,\tilde{T})=\sigma(\mu,T)$ for each such $\tilde{T}$. First choose an operator $\tilde{T}$ with characteristic polynomial $f$ in $\gamma_q(n)/(q^n-1)$ ways (since $\tilde{T}$ is cyclic, the algebra of operators commuting with $\tilde{T}$ is isomorphic to $\Fq[T]/(f)$ which has $q^n$ elements; subtract 1 for the (singular) zero operator). Once $\tilde{T}$ has been chosen, the number of choices for $W$ equals $\sigma(\mu,\tilde{T})$ by Corollary~\ref{cor:duality}. Therefore the total number of pairs $(W,\tilde{T})$ is given by  
  \begin{align*}
    \frac{\gamma_q(n)}{q^n-1}\; \sigma(\mu,T).
  \end{align*}
 On the other hand, first choose a $k$-dimensional subspace $W$ in ${n \brack k}_q$ ways. By Theorem \ref{thm:simplewithdefect}, the number of simple maps $T'$ defined on $W$ with defect dimensions $\mu$ is given by
  \begin{align*}
    q^{\sum_{j\geq 2}\mu_j^2}\;\gamma_q(k)\;\prod_{i\geq 1}{\mu_i\brack \mu_{i+1}}_q.
  \end{align*}
   By Lemma \ref{lem:numextensions}, the number of extensions of $T'$ to all of $\Fq^n$ with characteristic polynomial $f$ is given by $\prod_{j=k+1}^{n-1}(q^n-q^j)$. It follows that the total number of pairs $(W,\tilde{T})$ satisfying the stated property is given by
  \begin{align*}
  {n \brack k}_q  q^{\sum_{j\geq 2}\mu_j^2}\;\gamma_q(k)\;\prod_{i\geq 1}{\mu_i\brack \mu_{i+1}}_q\; \prod_{j=k+1}^{n-1}(q^n-q^j).
  \end{align*}
  By comparing the two counts, we obtain
  \begin{align*}
    \frac{\gamma_q(n)}{q^n-1}\; \sigma(\mu,T)={n \brack k}_q        q^{\sum_{j\geq 2}\mu_j^2}\;\gamma_q(k)\;\prod_{i\geq 1}{\mu_i\brack \mu_{i+1}}_q\; \prod_{j=k+1}^{n-1}(q^n-q^j).
  \end{align*}
The transitive action of the general linear group ${\rm GL}_n(\Fq)$ on subspaces of $\Fq^n$ of dimension $k$ yields, via the orbit-stabilizer theorem, the identity
  \begin{align*}
    \gamma_q(n)=\gamma_q(k)\gamma_q(n-k) q^{k(n-k)} {n \brack k}_q.
  \end{align*}
Substituting for $\gamma_q(n)$ from this identity, it follows that
  \begin{align*}
    \sigma(\mu,T)&=  (q^n-1) \frac{(q^n-q^{k+1})(q^n-q^{k+2})\cdots (q^n-q^{n-1})}{q^{k(n-k)}\gamma_q(n-k)} q^{\sum_{j\geq 2}\mu_j^2}\;\prod_{i\geq 1}{\mu_i\brack \mu_{i+1}}_q\\
   &= (q^n-1) \frac{(q^{n-k}-q)(q^{n-k}-q^{2})\cdots (q^{n-k}-q^{n-k-1})}{q^{k}(q^{n-k}-1)\cdots (q^{n-k}-q^{n-k-1})} q^{\sum_{j\geq 2}\mu_j^2}\;\prod_{i\geq 1}{\mu_i\brack \mu_{i+1}}_q\\
     &=  \frac{q^n-1}{q^{k}(q^{\mu_1}-1)}\; q^{\sum_{j\geq 2}\mu_j^2}\;\prod_{i\geq 1}{\mu_i\brack \mu_{i+1}}_q.
  \end{align*}
  The result now follows since $k=\sum_{j\geq 2}\mu_j$.
\end{proof}

\begin{corollary}\label{cor:split}
If $\F_{q^{md}}=\Fq(\alpha)$, then the number of $m$-dimensional $\alpha$-splitting subspaces of $\F_{q^{md}}$ is given by
  \begin{align*}
\frac{q^{md}-1}{q^{m}-1}q^{m(m-1)(d-1)}.
  \end{align*}
\end{corollary}
  
\begin{proof}
  Since $\alpha$ generates the extension $\F_{q^{md}}/\Fq$, the map $T_\alpha$ defined on $\F_{q^{md}}$ by $T_\alpha(x)=\alpha x$ is linear with characteristic polynomial equal to the minimal polynomial of $\alpha$ over $\Fq$. Therefore, the number of $m$-dimensional splitting subspaces in this case is given by $\sigma(\mu,T_\alpha)$, where $\mu=(m^d)$ denotes the rectangular partition with $d$ parts equal to $m$. By Theorem \ref{thm:main},  
\begin{align*}
&  \sigma((m^d),T_\alpha)=\frac{q^{md}-1}{q^m-1}q^{m(m-1)(d-1)}.\qedhere
\end{align*}
\end{proof}
\section{Acknowledgments}
This work was partially supported by an Indo-Russian project, grant number DST/INT/RUS/RSF/P41/2021 awarded by the Department of Science and Technology.
\printbibliography  
\end{document}